\documentclass[a4,10pt,reqno,fleqn]{amsart}

\title[Closed Lie ideals of tensor products of $C^*$-algebras II]{On
  closed Lie ideals of certain tensor products of $C^*$-algebras II}
\author[V. P. Gupta, R. Jain and B. Talwar]{Ved Prakash Gupta, Ranjana
  Jain and Bharat Talwar}

\address{Ved Prakash  Gupta, School of Physical Sciences, Jawaharlal Nehru University, New
  Delhi} \email{vedgupta@mail.jnu.ac.in}

\address{Ranjana Jain,
  Department of Mathematics, University of Delhi, Delhi}
\email{rjain@maths.du.ac.in}

\address{Bharat Talwar, Department of
  Mathematics, University of Delhi, Delhi}
\email{btalwar.math@gmail.com}

\thanks{The first named author was supported partially by a UPE-II
  project (with Id 228) of Jawaharlal Nehru University, New Delhi and
  the third named author was supported by a Junior Research Fellowship
  of CSIR with file number 09/045(1442)/2016-EMR-I}

\usepackage[all]{xy}
\usepackage{amsmath,amsfonts,amssymb,amsthm,a4,hyperref}
\usepackage{mathrsfs}

\usepackage[margin=3.07cm]{geometry}
\usepackage{hyperref}
\usepackage{cleveref}
\usepackage{setspace}
\usepackage{enumitem}
\usepackage{verbatim}

\newtheorem{theorem}{\bf Theorem}[section]
\newtheorem{lemma}[theorem]{\bf Lemma}
\newtheorem{proposition}[theorem]{\bf Proposition}
\newtheorem{corollary}[theorem]{\bf Corollary}
\newtheorem{remark}[theorem]{\bf Remark}
\numberwithin{equation}{section}

\newcommand{\norm}[1]{\left\Vert #1 \right\Vert}
\newcommand{\oop}{\widehat\otimes}
\newcommand{\seq}{\subseteq}
\newcommand{\oh}{\otimes^h}
\newcommand{\ota}{\otimes^{\alpha}}
\newcommand{\omin}{\otimes^{\min}}
\newcommand{\obp}{\otimes^\gamma}

\newcommand{\ra}{\rightarrow}
\newcommand{\ot}{\otimes}

\newcommand{\C}{\mathbb{C}}

\newcommand{\M}{\mathbb{M}}

\newcommand{\ol}{\overline}
\newcommand{\mcal}{\mathcal}
\begin{document}



\maketitle

\begin{abstract}
 We identify all closed Lie ideals of $A \ota B$ and $B(H) \ota B(H)$,
 where $\ota$ is either the Haagerup tensor product, the Banach space
 projective tensor product or the operator space projective tensor
 product, $A$ is any simple $C^*$-algebra, $B$ is any $C^*$-algebra
 with one of them admitting no tracial states, and $H$ is an infinite
 dimensional separable Hilbert space. Further, generalizing a result
 of Marcoux, we also identify all closed Lie ideals of $A\omin B$,
 where $A$ is a simple $C^*$-algebra with at most one tracial state
 and $B$ is any commutative $C^*$-algebra.
\end{abstract}
\vspace*{-4mm}

\section{Introduction}
 A complex associative algebra $A$ inherits a canonical Lie algebra
 structure via the Lie bracket $[x, y]:= xy -xy$. A subspace $L$ of
 $A$ is said to be a Lie ideal if $[L, A]\subseteq L$.

There exists extensive literature devoted to the study of Lie
ideals of pure algebras and of operator algebras (see, for instance,
\cite{bks, fms, fm, her, hms, mar, mie, mur} and the references
therein). However, apart from the results established in \cite{bks}
and \cite{mar}, not much is known about the structure of closed Lie
ideals of tensor products of operator algebras. Improvising their
techniques, some work on the study of closed Lie ideals of certain
tensor products of $C^*$-algebras was taken up in \cite{GJ1}.  This
article is a continuation of the same theme. Here is a quick overview
of the highlights of this article.

It was proved in \cite{GJ1} that {if $\ota$ is the Haagerup tensor
  product or the operator space projective tensor product} and $H$ is
an infinite dimensional separable Hilbert space, then the Banach
algebra $B(H) \ot^{\alpha} B(H)$ contains only one non-trivial central
closed Lie ideal, namely $\C (1 \ot 1)$, and that every non-central
closed Lie ideal of $B(H) \ot^{\alpha} B(H)$ contains the product
ideal $K (H) \ot^{\alpha} K(H)$. In this article, based on some recent
progress made in \cite{GJ2}, we include the study of closed Lie ideals
of the Banach space projective tensor product $A \obp B$ of $C^*$-algebras
$A$ and $B$, as well. Improving upon above  result of \cite{GJ1}, we go
one step ahead and identify all closed Lie ideals of $B(H)
\ot^{\alpha} B(H)$ as follows:
\vspace*{1mm}

\noindent {\bf \Cref{bh-ota-bh}}: {\em Let $H$ be an infinite
  dimensional separable Hilbert space. Then, for $\ota = \oh, \obp$ or
  $\oop$, the Banach algebra $B(H) \ota B(H)$ contains only 11
  distinct closed Lie ideals, each of which is either a
  closed ideal or is a subspace of the form $\C(1\ot 1) + K$ for some closed
  ideal $K$ in $B(H) \ota B(H)$.}
 \vspace*{1mm}

In \cite[Theorem 4.16]{GJ1}, it was shown that if $A$ and $B$ are
simple unital $C^*$-algebras with one of them admitting no tracial
functionals, then the only proper non-trivial closed Lie ideal of $A
\oh B$ or $A \oop B$ is $\C (1 \ot 1)$. We generalize this result  to
the following:
\vspace*{1mm}

\noindent {\bf \Cref{thm1}}: {\em
 Let $A$ be a simple $C^*$-algebra and $B$ be any $C^*$-algebra with
  one of them admitting no tracial states. If $\ota =
  \oh, \obp$ or $\oop$, then a subspace $L$ of $A \ota B$ is a closed Lie
  ideal if and only if it is of the form $$L = \left\{
  \begin{array}{ll}
    \overline{A \ot J + \C 1 \ot S}, & \ \text{if}\ A \ \text{is
      unital, and} \\ A \ota J, &
    \ \text{if}\ A \ \text{is non-unital},
  \end{array} \right. $$
  for some closed ideal $J$ in $B$ and some  subspace $S$ of  $N(J)$.
}

\vspace*{1mm}

 As a consequence, we deduce the following: 
\vspace*{1mm}

\noindent {\bf \Cref{A-oh-BH}}:{\em Let $A$ be a simple $C^*$-algebra,
  $H$ be an infinite dimensional separable Hilbert space and $\ota =
  \oh, \obp$ or $\oop$. If $A$ is unital, then $A \ota B(H)$ has only
  three non-trivial proper closed Lie ideals, namely, $\C (1 \ot 1)$,
  the closed ideal $A \ota K(H)$ and the closed subspace $ A \ota K(H)
  + \C(1 \ot 1)$. And, if $A$ is non-unital, then $A \ota B(H)$ has
  only one non-trivial proper closed Lie ideal, namely, the closed
  ideal $A \ota K(H)$.}
\vspace*{1mm}

In \cite{GJ1}, generalizing a result of \cite{mar}, all closed Lie
ideals of $A \omin B$ where identified for any simple unital
$C^*$-algebra $A$ with at most one tracial state and for any unital
commutative $C^*$-algebra $B$. Herein, with the help of a suitable
version of Tietze Extension Theorem for functions vanishing at
infinity, we prove that similar characterization holds even if $A$ and
$B$ are both non-unital.
\vspace*{1mm}

\noindent {\bf \Cref{A-omin-B}}: {\em  Let $A$ be a simple  $C^*$-algebra  and $B$ be a commutative $C^*$-algebra.
 \begin{enumerate}
   \item If $A$ is unital and admits at most one tracial state, then a  subspace $L$ of $A \omin B$ is a
closed Lie ideal if and only if
$$  L = \left\{
 \begin{array}{ll}
      \overline{sl(A) \ot J} + \C 1 \ot S, & \mathit{if}\ \mcal{T}(A)
     \neq \emptyset,\ \mathit{and} \\ \overline{A \otimes J + \C 1 \ot
       S}, & \mathit{if}\ \mcal{T}(A) = \emptyset,
 \end{array}
 \right.    $$
 for some closed ideal  $J$ and closed subspace $S$ in $B$.
\item If $A$ is non-unital with 
  $\mcal{T}(A) = \emptyset$ , then a subspace $L$ of $A \omin B$ is a
  closed Lie ideal if and only  $L$ is a closed ideal.
 \end{enumerate}
} 
\vspace*{1mm}

\noindent {\bf Comments:} Recall that, apart from being injective in their
respective categories, there is not much common between the tensor
produts $\oh$ and $\omin$; and, on the other hand, $\oop$ and $\omin$
have hardly anything similar to write about. Still, it is interesting to note the
striking similarity between the structure of $N(A \ota J)$ appearing
in \Cref{NK-eqn-2} and that of $N(A \omin J)$ that appears in
\Cref{6A}.  This further yields an unmissable analogy between the
closed Lie ideal structures of $A \ota B$ and of $A \omin C_0(X)$,
obtained in \Cref{thm1} and \Cref{A-omin-B}, respectively.

\section{Preliminaries}
 Let us first fix some notations and conventions, and recall some
terminologies that we shall adhere to. All algebras considered in this
article will be assumed to be associative over the base field
$\C$.  For subspaces $X$ and $Y$ of an algebra $A$, $[X, Y]$ denotes the
subspace generated by all commutators $[x, y]:=xy - yx$, $x \in X$, $y
\in Y$. A subspace $L$ of an algebra $A$ is called a {\em Lie ideal}
if $[L, A] \subseteq L$. If the algebra $A$ is unital, its unit will
be denoted by $1_A$ or simply by $1$ if there is no ambiguity. Any
ideal of $A$, the susbpace $\C 1_A$ and, more generally, the center
$\mcal{Z}(A)$ of $A$ are clearly Lie ideals.

For a $C^*$-algebra $A$, a {\em tracial state} $\varphi$ on $A$
is a positive linear functional of norm one satisfying the tracial
property $\varphi (ab) = \varphi(ba) $ for all $a, b \in A$, and
$\mathcal{T} (A)$ denotes the set of all tracial states on $A$. A {\em
  tracial functional} on a Banach algebra $A$ is a non-zero continuous
linear functional on $A$ satisfying the tracial property as above. 

\begin{remark}\label{tf-ts}
   For a $C^*$-algebra $A$, it is known that $\ol{[A,A]} = A$ if and
   only if $\mcal{T}(A) =\emptyset$ - \cite[Theorem 2.9]{cnp}.  In
   particular, as a consequence of Hahn-Banach Theorem, a
   $C^*$-algebra $A$ admits a tracial functional if and only if it
   admits a tracial state.
\end{remark}

For two vector spaces $V$ and $W$, their algebraic tensor product will
be denoted by $V \otimes W$. If $V$ and $W$ are algebras, then so is
$V \ot W$ with respect to the canonical product satisfying $(u \ot w)
(v \ot z) = uv \ot wz$ for all $u, v \in V$ and $w, z \in W$.

For any two Banach algebras $A$ and $B$, an {\em algebra norm} on $A
\ot B$ is a norm $\| \cdot \|_{\alpha}$ that satisfies $\| w z
\|_{\alpha} \leq \| w \|_{\alpha} \|z\|_{\alpha}$ for all $w, z \in A
\ot B$. For every such algebra norm, $A \ot^{\alpha} B$, the
completion of $A \ot B$ with respect to the norm $\| \cdot
\|_{\alpha}$, is a Banach algebra. Further, $\| \cdot\|_\alpha $ is a
  {\em cross norm} on $A \ot B$ if $\|a \ot b \|_\alpha= \|a\| \|b\|$
  for all $a \in A, b\in B$.  We will be mainly concerned with the
  minimal $C^*$-tensor norm $\| \cdot \|_{\min}$, the Haagerup tensor
  norm $\| \cdot \|_h$, the operator space projective tensor norm $\|
  \cdot \|_{\wedge}$ and the Banach space projective tensor norm $\|
  \cdot \|_{\gamma}$ (see \cite{ER} for details), all four of which
  are cross norms. In fact, for any pair of $C^*$-algebras $A$ and
  $B$, $A \omin B$ is a $C^*$-algebra, $A \oop B$ and $A \obp B$ are
  Banach $*$-algebras whereas the canonical involution on the Banach
  algebra $A \oh B$ is not an isometry (see \cite{ble, ER}). For the
  sake of convenience, we include a proof of the following elementary
  result, which will be used ahead.

  \begin{lemma}\label{A-ot-Mn-closed}
  Let $A$ be a Banach algebra with a bounded approximate
  identity. Then, for any cross algebra norm $\|\cdot\|_\alpha$ and
  any closed subspace $B$ of $A$, the space $B \ot \M_k$ is complete.
\end{lemma}
\begin{proof}  
  Let $\{e_{\lambda}\}$ be a bounded approximate identity of $A$ with
  $\sup_\lambda \|e_\lambda\| = M< \infty$.  Let $\{b_n = \sum_{ij}
  b_{ij}^{(n)} \ot E_{ij}\}$ be a Cauchy sequence in $B \ot \M_k$,
  where $\{E_{ij}\}$ is the standard system of matrix units of
  $\M_k$. For every $1 \leq i, j \leq k$, we have \begin{eqnarray*} \|
    e_{\lambda} (b^{(n)}_{ij}-b^{(m)}_{ij}) e_{\lambda} \| & = & \|
    e_{\lambda} ({b^{(n)}_{ij}-b^{(m)}_{ij}}) e_{\lambda} \|\,
    \|E_{ii}\| \\ & = &
    \|\big(e_{\lambda}{(b^{(n)}_{ij}-b^{(m)}_{ij})e_{\lambda}\big) \ot
      {E_{ii}}}\|_{\alpha} \\ & = & \|{(e_{\lambda}\ot E_{ii})
      (b_n-b_m) (e_{\lambda} \ot E_{ji})}\|_{\alpha} \\ & \leq &
    \|{e_{\lambda} \ot E_{ii}}\|_{\alpha}\, \|{b_n-b_m\|_{\alpha} \,
      \|{e_{\lambda}\ot E_{ji}}}\|_{\alpha}\\ & \leq &
    \|b_n-b_m\|_{\alpha} \, M^2 .
\end{eqnarray*}
  Since $e_{\lambda} a e_{\lambda} \rightarrow a$ for every $a \in A$, we obtain
  $$ \| (b^{(n)}_{ij}-b^{(m)}_{ij}) \| = \lim_\lambda \| e_{\lambda}
  (b^{(n)}_{ij}-b^{(m)}_{ij}) e_{\lambda} \| \leq \|b_n-b_m\|_{\alpha}
  \, M^2 ,$$ which shows that $\{b_{ij}^{(n)}\}_n$ is a Cauchy
  sequence and also that each co-ordinate map $B \ot \M_k \ni \sum_{ij}
  b_{ij} \ot E_{ij} \ra b_{ij} \in B$ is continuous.  Suppose
  $b^{(n)}_{ij} \longrightarrow b_{ij} \in B$, $1 \leq i, j \leq k$ and let $b=\sum_{ij}
  b_{ij} \ot E_{ij}$. Then, \[ \norm{b_n-b}_{\alpha} = \Big\|
  \sum_{i,j} (b_{ij}^{(n)} - b_{ij}) \ot E_{ij}\Big\|_{\alpha} \leq
  \sum_{i,j} \big\|(b_{ij}^{(n)} - b_{ij}) \big\| \|{E_{ij}}\|
  \longrightarrow 0. \]
\end{proof}

  It is known that the $C^*$-minimal tensor product is injective and
  so is the Haagerup tensor product (see \cite{sinc, ER}). So, for closed
  subspaces $E_0$ and $F_0$ in operator spaces $E$ and $F$,
  respectively, $E_0 \oh F_0$ can be identified isometrically with the
  closed subspace $\overline{E_0 \ot F_0}^{h}$ of the operator space $
  E \oh F$. On the other hand, neither the Banach space nor the
  operator space projective tensor product is injective.  But in some
  cases they observe extremely useful forms of partial injectivity.

\begin{remark}\label{I-oop-J} For closed ideals $I$ and $J$ of
$C^*$-algebras $A$ and $B$, respectively, it is known (\cite[Theorem 5]{kum}) that $I \oop J$
can be identified algebraically and isometrically with the product ideal $\ol{I \ot J}$
  of the Banach $*$-algebra $ A
  \oop B$.   In view of this identification, we shall consider $I \oop J$ as a
  closed ideal of $A \oop B$.
\end{remark}

In the Ph.D. thesis of the second named author, it was observed that
the situation is even better for $\obp$ - see \cite[$\S 2$]{GJ2} for details.

\begin{remark}\label{obp-injective}
  For closed $*$-subalgebras $A_1$ and $B_1$ of $C^*$-algebras $A$ and
  $B$, respectively, it is known (\cite[Theorem 2.6]{GJ2}) that $A_1
  \obp B_1$ can be identified $*$-isomorphically and isometrically
  with the closed $*$-subalgebra $\ol{A_1 \ot B_1}$ of the Banach
  $*$-algebra $ A \obp B$.  In view of this identification, we shall
  consider $A_1 \obp B_1$ as a closed $*$-subalgebra of $A \obp B$.
\end{remark}

We will have occassions to appeal to the following  elementary observations.
\begin{lemma}\label{commutators-tf}\label{commutators}
Let $A$ and $B$ be Banach algebras. Suppose $B$ is unital
and $A$ admits no tracial functionals. If $\| \cdot \|_{\alpha}$ is any
algebra norm, then
$$\ol{[A \ot B, A \ot B]} = \ol{[A \ot^{\alpha} B, A \ot^{\alpha} B]} = A \ot^{\alpha} B.$$ In
particular, $A \ota B$ does not have any tracial functionals.
\end{lemma}
  \begin{proof}
Since $A$ has no tracial functionals, $\ol{[A, A]} = A$. In
particular, $\ol{[A, A]} \ot B$ is dense in $A \ota B$ and for each
$a, a' \in A$ and $b \in B$, we have $[a, a']\ot b = [a \ot b, a' \ot
  1]$, so that $$ A \ota B = \ol{[A, A] \ot B} \subseteq \ol{[A \ot B,
    A \ot B]} \subseteq A \ota B.$$ In particular, by Hahn-Banach
Theorem, $A \ota B$ admits no tracial functionals.
\end{proof}

  For $C^*$-algebras, thanks to continuous calculus, we can drop unitality from the hypothesis.
\begin{lemma}\label{commutators-c-star}
Let $A$ and $B$ be $C^*$-algebras and suppose $B$ admits no tracial
states. If $\| \cdot \|_{\alpha}$ is any algebra norm, then
$$\ol{[A \ot B, A \ot  B]} = \ol{[A \ot^{\alpha} B, A \ot^{\alpha} B]} = A \ot^{\alpha} B.$$ In
particular, $A \ota B$ does not have any tracial functionals.
\end{lemma}
\begin{proof}
 Since $B$ has no tracial states, $\ol{[B, B]} = B$ - see
 \Cref{tf-ts}. In particular, $A \ot {[B, B]} $ is dense in $A \ota
 B$. Now, for any $a \geq 0$ in $A$ and $b, b' \in B$, we see that
\[
a \ot [b, b'] = a \ot b b' - a \ot b' b = [a^{1/2} \ot b, a^{1/2} \ot
  b'] \in [A \ot B, A \ot B] \subseteq A \ot B.
\]
 This yields
\[
A \ota B = \ol{A \ot [B, B]} \subseteq \ol{[A \ot B, A \ot B]} \subseteq A \ota B,
\]
and we have the desired equality.
 
\end{proof}

Analogous to \Cref{commutators}, for pure algebras, we have the following:
\begin{lemma} \label{1A}
  Let $C$ and $D$ be algebras and suppose $D$ is
  unital. Then, \begin{equation}\label{commutator} [c, c'] \otimes
    [d,d'] \in [C\otimes D, C\otimes D ]
    \end{equation} for all $c, c'
  \in C$ and $d, d' \in D$. In particular, if both $C$ and $D$ are
  spanned by their commutators then so is $C \otimes D$, i.e.,
  $[C\otimes D, C\otimes D ] = C\otimes D.$
\end{lemma}
\begin{proof}
A straight forward calculation yields the equality
$$2[c\otimes 1_D,c' \otimes d d']-[c \otimes d, c' \otimes d']-[c
  \otimes d', c' \otimes d]= [c, c']\otimes [d, d'],$$ thereby
implying \Cref{commutator}. The second assertion  follows readily from
\Cref{commutator}. \end{proof}

Unlike $\M_n$, for an infinite dimensional Hilbert space $H$, $B(H)$
is unital and is spanned by its commutators (see \cite[page
  198]{hal}). We can immediately deduce the following:
\begin{corollary} \label{2A}
  For an infinite dimensional Hilbert space $H$, we have
  \begin{enumerate}
    \item $[B(H)\otimes B(H), B(H)\otimes B(H) ] = B(H)\otimes B(H).$
\item for any algebra norm $\|\cdot \|_{\alpha}$, the Banach algebra
  $B(H)\otimes^\alpha B(H)$ does not have any  tracial functional.
\end{enumerate}
  In particular,
  \[
  \ol{[B(H)\ot B(H), B(H)\ot B(H) ]} = \ol{[B(H)\ota B(H), B(H)\ota
      B(H) ]} =B(H)\ota B(H).
  \]
\end{corollary}

\section{Closed Lie ideals of $A \oh B, A \obp B$ and $A \oop B$}
From the existing literature on Lie ideals, it is evident that one of
the few tools available and well exploited to study Lie ideals, is the
notion of {\em Lie normalizer}, which was introduced by Murphy
(\cite{mur}) in 1984. (At this moment, we pause to remark that this is not
the standard terminology; since we are not aware of one, we are giving
it this name, hoping that it doesn't conflict with some existing
usage.)  For any subspace $S$ of an algebra $A$, its {\em Lie
  normalizer} is given by $$N(S) = \{ a \in A : [a, A] \subseteq
S\}.$$ Note that the Lie normalizer is different from the usual
normalizer $N_A(S)$, which is given by $N_A(S) = \{ a \in A: [a, S]
\subseteq S\}$. If $L$ is a Lie ideal of $A$, then
$N(L)$ is a  subalgebra as well as a Lie ideal of $A$
(\cite[Proposition 2.2]{bks}).  As usual, $\mcal{Z}(A)$ will denote
the center of an algebra $A$; note that
$N((0)) = \mcal{Z} (A)$.

The importance of the notion of Lie normalizer is reflected from the
extremely useful fact that if $A$ is a suitable Banach algebra
(\cite[Theorem 4.5]{GJ1}) or an arbitrary $C^*$-algebra
(\cite[Corollary 5.26 and Theorem 5.27]{bks}), then a closed subspace
$L$ of $A$ is a Lie ideal if and only if $$ \ol{[I, A]} \subseteq L
\subseteq N(\ol{[I,A]})$$ for some closed ideal $I$ in $A$.  All
ideals appearing in this article will be assumed to be two-sided.

The following elementary observation (which kind of appears on Page
3328 of \cite{hms} as well) turns out to be very useful in the
identification of  closed Lie ideals of certain tensor products of
$C^*$-algebras.
\begin{lemma}\label{NI}
Let $I$ be an ideal in an algebra $A$. Then, we have $$
N(I)=\pi^{-1}\left(\mcal{Z}(A/I)\right),
$$ where $\pi$ is the canonical quotient map from $A$ onto $A/I$. In
particular, $N(I)$ is always a subalgebra of $A$. Moreover, if $A$ is
a $*$-algebra, then so is $N(I)$ for any $*$-ideal $I$.
\end{lemma}

This immediately yields the following useful consequences.
\begin{corollary}\label{NI-banach}
Let $A$ be a unital Banach algebra with a unique non-trivial closed
ideal $I$. Then, $N(I) = \C 1 + I$. In particular, $\C 1 + I$ is a
closed Lie ideal.
\end{corollary}
\begin{proof}
Since $A/I$ is a simple unital Banach algebra, we have $\mcal{Z}(A/I)
= \C (1 + I)$, and hence $$N(I) =\pi^{-1} \left( \mcal{Z}(A/I) \right)
= \pi^{-1}(\C (1 + I)) = \C 1 + I.$$ Since $\C (1 + I)$ is closed in
$A/I$, so is $\pi^{-1}(\C (1 + I)) = \C 1 + I$ in $A$. \end{proof}

\begin{corollary}\label{4A}
  For an infinite dimensional separable Hilbert space $H$, we have
$$ N(K(H))= \C 1 + K(H),$$ where $K(H)$ denotes the closed ideal of $B(H)$
  consisting of compact operators on $H$.
\end{corollary}
This yields an operator algebraic proof of the following classical 
(operator theoretic) result of Fong-Miers-Sorour (\cite{fms}):
\begin{corollary}
  Let $H$ be an infinite dimensional separable Hilbert space $H$. Then
  the only non-trivial closed Lie ideals of $B(H)$ are $ \C 1, K(H)$
  and $ \C 1 + K(H)$.
 \end{corollary}
\begin{proof}
Clearly $\C 1, K(H)$ and $\C 1 + K(H)$ are closed Lie ideals of
$B(H)$. Conversely, if $L$ is a closed Lie ideal of $B(H)$, then by
\cite[Theorem 5.27]{bks}, there exists a closed ideal $I$ of $B(H)$
such that $\ol{[I, B(H)]} \subseteq L \subseteq N(\overline{[I, B(H)
]})$. Since $B(H) = [B(H), B(H)]$, by \cite[Lemma 4.2]{GJ1}, we have
$\ol{[I, B(H)]} = I$. Thus, $I \subseteq L \subseteq N(I)$ and
depending upon the three possible choices of $I$, namely, $(0), K(H)$ and
$B(H)$, we see that $L$ has only above possibilities.
  \end{proof}

We now identify Lie normalizer of sum of product ideals of $A\oh B$.

\begin{proposition}\label{NK}
   Let $I$ and $J$ be closed ideals in $C^*$-algebras $A$ and
   $B$, respectively. Then for $K = A \oh J + I \oh B$, we have 
   \begin{equation}\label{NK-eqn} N(K) = K+ N(I) \oh N(J).
   \end{equation} And, $N(I \oh J) \subseteq N(I) \oh N(J)$.
    \end{proposition}

\begin{proof}
  It is known that $K$ is a closed ideal in $C:=A \oh B$, the
  quotient map $\pi_I \oh \pi_J: A \oh B \ra (A/I) \oh (B/J)$ has
  kernel $K$ and it induces an isometric algebra isomorphism
  $\ol{\pi_I \oh \pi_J} : C/K \ra (A/I) \oh (B/J)$ - see
  \cite[Corollary 2.6]{sinc}. In particular, $\mcal{Z}(C/K) = \Big(\ol{\pi_I \oh
    \pi_J}\Big)^{\ -1}\Big(\mcal{Z}\big( (A/I) \oh (B/J)\big)\Big)$. Also, by
  \cite[Theorem 2.13]{sinc}, we have $\mcal{Z}\big( (A/I) \oh
  (B/J)\big) = \mcal{Z}( A/I) \oh \mcal{Z}(B/J)$ canonically. So,
  \begin{eqnarray*}
      \mcal{Z}(C/K) & = &\Big(\ol{\pi_I\oh \pi_J}\Big)^ {\ -1} \Big(
      \mcal{Z}(A/I) \oh \mcal{Z}(B/J) \Big) \\
      & = &     \Big({(\pi_I\oh \pi_J)}^ {\ -1} \left( \mcal{Z}(A/I) \oh
      \mcal{Z}(B/J) \right) \Big)/ K \\
      & = & \Big( \big(\pi_I^{-1} (0 + I)
           \oh B \big) + \big( A \oh \pi_J^{-1}(0 + J)\big) +
           \big(\pi_I^{-1} (\mcal{Z}(A/I))
           \oh \pi_J^{-1}(\mcal{Z}(B/J)) \big)\Big) / K \\ & = & \Big(I
           \oh B + A\ota J + N(I) \oh N(J)\Big)/ K
           \\ & = & \Big( N(I) \oh N(J) + K \Big)/ K,
  \end{eqnarray*}
  where the third equality follows from \cite[Theorem 2.4]{sinc},
  which also yields that $N(I) \oh N(J) + K$ is closed in $A
  \oh B$. In particular, if $\pi_K: C \ra C/K$ is the natural quotient
  map, then we obtain \begin{eqnarray*} \pi_K^{-1}(\mcal{Z}(C/K)) & =
    & \pi_K^{-1} \Big( \big(N(I) \oh N(J) + K\big) /K \Big)
    \\ & = & N(I) \oh N(J) + K,
  \end{eqnarray*}
  and by \Cref{NI}, we have the desired relation for $N(K)$.
  \vspace*{1mm}

  Next, substituting $J = (0)$ in $K$, we obtain $N (I \oh B) = I \oh
  B + N(I) \oh \mcal{Z}(B)$ and likewise $N( A \oh J) = A \oh J +
  \mcal{Z}(A) \oh N(J)$. Thus,
  \begin{eqnarray*}
    N(I \oh J) & \subseteq & N( A \oh J) \cap N (I \oh B) \\ & =
    &\left( A \oh J + \mcal{Z}(A) \oh N(J) \right) \cap \left( I \oh B
    + N(I) \oh \mcal{Z}(B) \right) \\
    & \subseteq & (A \oh N(J)) \cap (N(I) 
    \oh B) \\ & = & N(I) \oh N(J),
  \end{eqnarray*}
  where injectivity of $\oh$ has been used extensively and the
  equality in the last step follows from the fact that $(E_1 \oh F_1 )
  \cap (E_2 \oh F_2) = (E_1 \cap E_2) \oh (F_1 \cap F_2)$ for closed
  subspaces $E_i, F_i$, of $C^*$-algebras $A_i, B_i$, $i =1, 2$,
  respectively (see \cite[Corollary 4.6]{smi}).
\end{proof}

\begin{remark}
In general, $N(I) \oh N(J) \neq N(I \oh J)$. For instance, if we take
$I= (0)$ and $J$ to be some non-trivial closed ideal in unital
$C^*$-algebras $A$ and $B$, respectively, then we have $$N(I \oh
J)=N((0))= \mathcal{Z}(A \oh B) = \mathcal{Z}(A) \oh \mathcal{Z}(B),$$
where the last equality follows from \cite[Theorem 2.13]{sinc},
whereas $N(I) \oh N(J)= \mathcal{Z}(A) \oh N(J) \supseteq
\mathcal{Z}(A) \oh (J + \mathcal{Z}(B)) $, which may very well be
strictly larger than $\mathcal{Z}(A) \oh \mathcal{Z}(B)$. {Taking
  $A=B=B(H)$ and $J= K(H)$ for an infinite dimensional separable
  Hilbert space $H$ provides one such instance.}
  \end{remark}

By \cite[Theorem 3]{jk08}, we just have an algebraic isomorphism
between $\mcal{Z}(A \oop B)$ and $\mcal{Z}(A) \oop \mcal{Z}(B)$ and an
analogue of \cite[Theorem 2.4]{sinc} for $\oop$ is not known.  In order to
establish the analagoue of expression (\ref{NK-eqn}) of Lie normalizer of a sum of
product ideals in $A \oop B$, we need some preparation.

 Recall, from \cite[Proposition 7.1.7]{ER}, that for operator spaces
 $V_i, W_i$, $i =1, 2$ and complete quotient maps $\varphi_i : V_i
 \rightarrow W_i$, $i = 1,2$, the tensor map $\varphi_1 \ot \varphi_2$
 extends to a complete quotient map $\varphi_1 \oop \varphi_2 : V_1
 \oop V_2 \ra W_1 \oop W_2$ and $K:=\ker(\varphi_1 \oop \varphi_2) =
 \ol{V_1 \ot \ker(\varphi_2) + \ker(\varphi_1) \ot V_2}$. Thus,
 $\varphi_1 \oop \varphi_2$ induces a bijective continuous map
 $\ol{\varphi_1 \oop \varphi_2} : (V_1 \oop V_2)/K \ra W_1 \oop
 W_2$. By Open Mapping Theorem, $\ol{\varphi_1 \oop \varphi_2}$
 becomes bicontinuous, i.e., $\Big(\ol{\varphi_1 \oop
   \varphi_2}\Big)^{-1}$ is also continuous. We will need the
 following analogue of \cite[Theorem 2.4]{sinc}.

 \begin{lemma}\label{ker-phi-ot-psi}
 Let $V_i, W_i$, $i =1, 2$ be operator spaces, $\varphi_i : V_i
 \rightarrow W_i$, $i = 1,2$ be complete quotient maps and  $E_i$ be
 a closed subspace of $W_i$ for $i=1,2$.  Then, \begin{equation}\label{ker-eqn}
   \Big(\ol{\varphi_1 \oop
   \varphi_2}\Big)^{-1} \big( \ol{E_1 \ot E_2}\big) = \Big(\ol{
   \varphi_1^{-1}(0) \ot V_2 + V_1 \ot \varphi_2^{-1}(0) +
   \varphi_1^{-1}(E_1) \ot \varphi_2^{-1}(E_2)}\Big)/K.
 \end{equation} 
  \end{lemma}
 \begin{proof}
   Let $x \in \ol{E_1 \ot E_2} $ and set $\beta=\ol{\varphi_1 \oop
     \varphi_2}$. Choose a sequence $\{x_n = \sum_{i=1}^{k_n} e_{i, n}
   \ot f_{i,n} \}_n$ in $ E_1 \ot E_2$ converging to $x$. For each
   pair $(i, n)$, where $ 1 \leq i \leq k_n$ and $n \geq 1$, fix
   $y_{i,n} \in \varphi_1^{-1}(\{e_{i,n}\})$ and $z_{i,n} \in
   \varphi_2^{-1}(\{f_{i,n}\})$. Then, $\beta^{-1}(x_n)=
   \sum_{i=1}^{k_n} (y_{i,n} \ot z_{i,n}) + K$, by very definition of
   $\beta$, so that $\beta^{-1}(x_n) \in \big(\varphi_1^{-1}(E_1) \ot
   \varphi_2^{-1}(E_2) + K\big)/K$ for all $n$. Since $\beta^{-1}$ is
   continuous, $ \beta^{-1}{x} = \lim_{n \rightarrow \infty}
   \beta^{-1}(x_n) \in \Big(\ol{\varphi_1^{-1}(E_1) \ot
     \varphi_2^{-1}(E_2) + K}\Big )/ K, $ thereby proving
   that $$\beta^{-1}\Big(\ol{E_1 \ot E_2}\Big) \subseteq
   \Big(\ol{\varphi_1^{-1}(E_1) \ot \varphi_2^{-1}(E_2) + K}\Big) / K
   .$$

Next, let $y +K \in \Big( \ol{\varphi_1^{-1}(E_1) \ot
  \varphi_2^{-1}(E_2) + K}\Big) /K$. Fix a sequence $\{y_n =
\sum_{i=1}^{k_n} y_{i,n} \ot z_{i,n} \}$ in $ \varphi_1^{-1}(E_1) \ot
\varphi_2^{-1}(E_2) $ such that $y_n+K \ra y + K$ in $(V_1 \oop
V_2)/K$. Then, $\beta(y_n+K)= \sum_{i=1}^{k_n} \varphi_1(y_{i,n}) \ot
\varphi_2(z_{i,n}) \in E_1 \ot E_2 $ for all $n$; so that $\beta(y +
K) = \lim_n \beta(y_n + K) \in \ol{E_1 \ot E_2}$ which shows that $y +
K \in \beta^{-1}(\ol{E_1 \ot E_2})$. \end{proof}

For Banach spaces $X_i, Y_i$, $i =1, 2$ and quotient maps $\varphi_i :
X_i \rightarrow Y_i$, $i = 1,2$, the tensor map $\varphi_1 \ot
\varphi_2$ extends to a quotient map $\varphi_1 \obp \varphi_2 : X_1
\obp X_2 \ra Y_1 \obp Y_2$ (see \cite[Proposition 2.5]{ryan}) and
 $K:=\ker(\varphi_1 \obp \varphi_2) = \ol{V_1 \ot \ker(\varphi_2) +
  \ker(\varphi_1) \ot V_2}$ - see \cite[Proposition 3.10]{GJ2}. Thus,
$\varphi_1 \obp \varphi_2$ induces a bijective continuous map
$\ol{\varphi_1 \obp \varphi_2} : (X_1 \obp X_2)/K \ra Y_1 \obp Y_2$.
Note that no specific property of $\| \cdot \|_{\wedge}$ was used in
the proof of \Cref{ker-phi-ot-psi}. So, the same proof works for
$\|\cdot \|_{\gamma}$ as well, and we have:

  \begin{lemma}\label{ker-phi-obp-psi}
 Let $X_i, Y_i$, $i =1, 2$ be Banach  spaces, $\varphi_i : X_i
 \rightarrow Y_i$, $i = 1,2$ be quotient maps and  $E_i$ be
 a closed subspace of $Y_i$, $i=1,2$.  Then, \begin{equation}\label{ker-eqn-obp}
   \big(\ol{\varphi_1 \obp
   \varphi_2}\big)^{-1} \big( \ol{E_1 \ot E_2}\big) = \Big(\ol{
   \varphi_1^{-1}(0) \ot V_2 + V_1 \ot \varphi_2^{-1}(0) +
   \varphi_1^{-1}(E_1) \ot \varphi_2^{-1}(E_2)}\Big)/K.
 \end{equation} 
  \end{lemma}

In fact,  \Cref{ker-phi-obp-psi} follows readily from \cite[Proposition
    3.10]{GJ2}, as well.
  
  \begin{lemma}\label{center-A-oop-B}
Let $A$ and $B$ be $C^*$-algebras. Then, $\mcal{Z}(A) \ot\mcal{Z}( B)
= \mcal{Z}(A \ot B) \subseteq \mcal{Z}(A \oop B)$ and $\ol{\mcal{Z}(A)
  \ot\mcal{Z}( B)} = \mcal{Z}(A \oop B)$.
\end{lemma}

\begin{proof}
Since $A \ot B$ is dense in $A \oop B$, we have $\mcal{Z}(A \ot B) \subseteq
\mcal{Z}(A \oop B)$.

By \cite[Theorem 3]{jk08}, the identity map on $\mcal{Z}(A \ot B) =
\mcal{Z}(A) \ot \mcal{Z}(B)$ extends to a bicontinuous algebraic
isomorphism, say, $\theta$ from $\mcal{Z}(A \oop B)$ onto $\mcal{Z}(A)
\oop \mcal{Z}(B)$. Let $y \in \mcal{Z}(A \oop B)$. Then, there
exists a sequence $\{y_n\}$ in $\mcal{Z}(A) \ot \mcal{Z}(B)=
\mcal{Z}(A \ot B)$ such that $\{y_n = \theta(y_n)\}$ converges to
$\theta(y)$ in $\mcal{Z}(A) \oop \mcal{Z}(B)$. In particular, $y_n
=\theta^{-1}(y_n) \ra y$ in $ \mcal{Z}(A \oop B)$. \end{proof}

\begin{remark}\label{center-A-obp-B}
For any two $C^*$-algebras $A$ and $B$, we have $\mcal{Z}(A) \obp
\mcal{Z}( B) = \mcal{Z}(A \obp B)$ canonically and isometrically
$*$-isomorphically - see \cite[Theorem 5.1]{GJ2}. 

Also, from \Cref{NI}, we see that $N(I)$ is a closed $*$-subalgebra of $A$
for any closed ideal $I$ in a $C^*$-algebra $A$. So, by \cite[Theorem
  2.6]{GJ2} (see \Cref{obp-injective} above), we can identify $N(I)
\obp N(J)$ with the closed $*$-subalgebra $\ol{N(I) \ot N(J)}$ of $A
\obp B$.
  \end{remark}

\begin{proposition}\label{NK-oop}
   Let $I$ and $J$ be closed ideals in $C^*$-algebras $A$ and $B$,
   respectively, and $\ota = \oop$ or $\obp$. Then, for $K = A \ota J
   + I \ota B,$ we have
   \begin{equation}\label{NK-eqn-oop}
     N(K) = \ol{K+ N(I) \ot N(J)}.
   \end{equation}
    \end{proposition}

\begin{proof}
 By \cite[Proposition 7.1.7]{ER} and \cite[Proposition 3.10]{GJ2},
 $\ker (\pi_I \ota \pi_J) = \ol{K} $ and, by \cite[Lemma 2]{jk11} and
 \cite[Corollary 3.17]{GJ2}, $K$ is closed. So, $\pi_I \ota \pi_J$
 induces a continuous algebra isomorphism $\ol{\pi_I \ota \pi_J}$ from
 $ (A \ota B)/K $ onto $(A/I) \ota (B/J)$.  Thus, by
 \Cref{center-A-oop-B} and \Cref{center-A-obp-B}, we obtain
  \begin{eqnarray*}
    \mcal{Z}\big((A \ota B)/K\big) & = & \Big( \ol{\pi_I \ota
      \pi_J}\Big)^{-1}\Big(\mcal{Z}\big( (A/I) \ota (B/J) \big) \Big)
    \\ & = & \Big( \ol{ \pi_I \ota \pi_J}\Big)^{ -1}
    \Big(\ol{\mcal{Z}(A/I) \ot \mcal{Z}(B/J)}\Big) \\ & = &
    \Big(\ol{N(I) \ot N(J) + K}\Big)/K, \hspace*{30mm}
    (\text{by\ Lemmas}\  \ref{ker-phi-ot-psi},\, \ref{ker-phi-obp-psi}\ \&\   \ref{NI})
  \end{eqnarray*} and, if $\pi_K: A \ota B \ra (A \ota B)/K$ is the natural quotient
map, then this yields
\[
N(K) =   \pi_K^{-1}\Big(\mcal{Z}\big((A \ota B)/K \big)\Big) = \ol{N(I) \ot N(J) + K},
\]
as was asserted.
\end{proof}

\begin{corollary}\label{NK-cor-1}
Let $A$ and $B$ be  $C^*$-algebras and $I$ and $J$ be closed
ideals of $A$ and $B$, respectively, such that $\mcal{Z} (A/I)$ or
$\mcal{Z} (B/J)$ is one dimensional. If $\ota = \oh, \obp$ or $\oop$,
then for the closed ideal $K = A \ota J + I \ota B$ of $A \ota B$, we
have
\begin{equation}\label{NK-eqn-2}
  N(K)= \left\{
  \begin{array}{ll}
  \ol{N(I) \ot \C 1 + K} & \text{if } \dim (\mcal{Z} (B/J)) =
  1\ \text{and}\ B\ \text{is unital},\ \text{ and}\\ \ol{\C 1 \ot N(J)
    + K} & \text{if } \dim (\mcal{Z} (A/I)) =
  1\ \text{and}\ A\ \text{is unital}.
\end{array}
  \right.
\end{equation}
Further, if $A$ and $B$ are both unital and, $\mcal{Z} (A/I)$ and
$\mcal{Z} (B/J)$ are both one dimensional, then $N(K) =\C ( 1\ot 1) +
K$.
\end{corollary}

\begin{corollary}\label{NK-unique}
  Let $A$ and $B$ be
  unital $C^*$-algebras and $\ota = \oh, \obp $ or $\oop$.
  \begin{enumerate}
  \item If $I$ and $J$ are non-trivial maximal closed ideals of $A$
    and $B$, respectively, and $K = A \ota J + I \ota B$, then
\begin{equation}\label{NK-eqn-3}
  N(K) = \C (1\ot 1) + K.
\end{equation}
\item In fact, if  $I$ and $J$ are the unique
non-trivial proper closed ideals of $A$ and $B$, respectively, then
(\ref{NK-eqn-3}) holds for every closed ideal $K$ in $A \ota B$.
\end{enumerate}
\end{corollary}

\begin{proof}
  Since $I$ is maximal, $N(I) = \C1 + I$, by \Cref{NI}, and likewise,
  $N(J) =\C1 +J$. So, in view of \Cref{NK-cor-1}, we just need to
  prove the second part.

  If $I$ and $J$ are the unique non-trivial proper closed ideals of
  $A$ and $B$, respectively , then it is known that the only closed
  ideals of $A \ota B$ are $(0), I \ota J, A \ota J, I \ota B, A\ota J
  + I \ota B $ and $ A \ota B$ - see \cite[Theorem 5.3]{sinc},
  \cite[Theorem 3.16]{GJ2} and \cite[Theorem 3.4]{kr}. Since $A/I$ and
  $B/J$ are simple and unital, we have $\dim \mcal{Z}(A/I) = \dim
  \mcal{Z}(B/J) =1$. So, by \Cref{NK-cor-1} again, we obtain the
  desired expression for the Lie normalizers of the closed ideals $A
  \ota J$ and $I \ota B$.

Finally, for $I \ota J$, since $I\ota J$ can be identified with the
product ideal $\ol{I \ot J}$ of $A \ota J$ as well as of $I \ota B$
(see Remarks \ref{I-oop-J} and \ref{obp-injective}), we obtain
\[
N(I \ota J) \subseteq N(A \ota J)
\cap N(I \ota B) = \big( \C (1\ot 1) + A \ota J \big) \cap \big(
\C (1 \ot 1) + I \ota B\big).
\]
We claim that the last intersection equals $\C(1\ot 1) + I \ota J$
. Indeed, if $ \lambda_1{(1\ot 1)} + z = \lambda_2{(1\ot 1)} + w$ for
$z \in A \ota J$ and $w \in I \ota B$ then $z - w = (\lambda_2
-\lambda_1) (1\ot 1)$. Since $z - w$ is in the unique maximal ideal $A
\ota J + I \ota J$ of $A \ota B$, it can not be a non zero
scalar. Thus, $z = w$.
\end{proof}

A Lie ideal $L$ in an algebra $A$ is said to be central if $L
\subseteq \mcal{Z}(A)$. For $\ota = \oh$ and $\oop$, in \cite{GJ1}, it
was shown that the only non-zero central Lie ideal of
$B(H)\otimes^{\alpha} B(H)$ is $\C ( 1 \ot 1)$ and that every
non-central closed Lie ideal in $B(H)\otimes^{\alpha} B(H)$ contains
the product ideal $K(H) \ot^{\alpha} K(H)$. We can now identify all
closed Lie ideals of $B(H)\otimes^{\alpha} B(H)$ (including for $\ota =
\obp$), as a consequence of the following:
 
\begin{theorem}\label{A-ota-B}
Let $A$ and $B$ be unital $C^*$-algebras and suppose one of them
admits no tracial states. Let $\ota = \oh, \obp$ or $\oop$.  If each
of $A$ and $B$ contains a unique proper non trivial closed ideal, then
$A \ota B$ has precisely $11$ distinct closed Lie ideals, each of
which is either a closed ideal or a subspace of the form $\C(1\ot 1) +
K$ for some closed ideal $K$.
\end{theorem}
\begin{proof}
 Let $C: = A \ot^{\alpha} B$ and $L$ be a closed Lie ideal in
 $C$. Since $\| \cdot\|_h$, $\|\cdot\|_{\gamma}$ and $\| \cdot
 \|_{\wedge}$ are all algebra norms, $A \otimes^{\alpha} B$ does not
 admit any  tracial functional - see
 \Cref{commutators-c-star}. And, by \cite[Remark 4.14]{GJ1}, every closed
 ideal in $ A \otimes^{\alpha} B $ admits a quasi-central approximate
 identity.  Therefore, by \cite[Theorem 4.5]{GJ1}, there exists a
 closed ideal $K$ in $A\ota B$ such that $K \subseteq L \subseteq
 N(K)$.

 By \Cref{NK-unique}, we have $N(K) = \C (1 \ot 1) + K $ for $K \neq
 (0)$ and, by \Cref{NI}, we have
 \[
 N((0)) = \mcal{Z}(A \ota B) \cong \mcal{Z}(A) \ota \mcal{Z}(B),
 \]
 where, as seen above, the relationship between centers comes from
 \cite[Theorem 2.13]{sinc}, \Cref{center-A-oop-B} and
 \Cref{center-A-obp-B}. Then, since $A$ and $B$ contain unique
 non-trivial proper closed ideals and are unital, we have $\dim \mcal{Z}(A) = 1=  \dim
 \mcal{Z}(B)$ - see  \cite[Lemma 2.1]{ART}. Hence,  $N((0)) = \C(1 \ot 1).$

As recalled in \Cref{NK-unique}, $A \ota B$ has only 6 closed ideals,
namely, $(0), I \ota J, A \ota J, I \ota B, A\ota J + I \ota B $ and $
A \ota B$. Thus, there are only 11 closed Lie ideals in $A \ota B$ of
the form $K$ or $\C(1 \ot 1) +K$ for some closed ideal $K$.
\end{proof}

\begin{corollary}\label{bh-ota-bh}
Let $H$ be an infinite dimensional separable Hilbert space and
$\ot^{\alpha} =\oh, \obp$ or $\oop$. Then, the Banach algebra $B(H) \ota B(H)$
contains only $11$ distinct closed Lie ideals, each of which is either
a closed ideal or is a subspace of the form $\C(1 \ot 1) + K$ for some
closed ideal $K$ in $B(H) \ota B(H)$.
\end{corollary}

\noindent{\bf Examples of $C^*$-algebras with unique ideals.}
  Apart from $B(H)$ for an infinite dimensional  separable Hilbert space
  $H$, there are many other examples of $C^*$-algebras containing
  unique non-trivial closed ideals. We list few of them here: 
  \begin{itemize} \item For any non-unital
  simple $C^*$-algebra $A$, its unitization $\widetilde{A}$ has a
  unique closed ideal, namely, $A$.
  \item For a non-unital simple $C^*$-algebra $A$ belonging to a
    fairly large family, the corona $C^*$-algebra $M(A)/A$ is simple
    and unital, where $M(A)$ is the multiplier algebra of $A$ - see
    \cite{lin} and the references therein. In particular, for every
    such $A$, the only non-trivial proper closed ideal in $M(A)$ is
    $A$.
  \item There are many AF $C^*$-algebras with unique non-trivial
    closed ideals. For instance, by \cite[Theorem 3.3]{bra} (also see
    \cite[Theorem III.4.2]{dav}), the AF algebra with Bratteli diagram
    
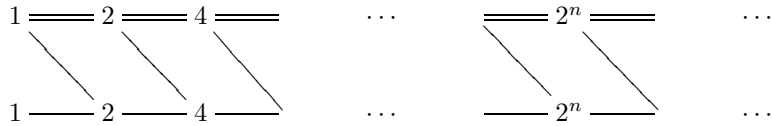
\[\hspace*{10mm}\xymatrix{ 1 \ar@{=}[r] \ar@{-}[rd] &
      2\ar@{=}[r] \ar@{-}[rd] & 4 \ar@{-}[rd] \ar@{=}[r] & & \cdots &
      \ar@{=}[r] \ar@{-}[rd] & 2^n \ar@{=}[r] \ar@{-}[rd] & & \cdots
      \\ 1 \ar@{-}[r]& 2\ar@{-}[r] & 4 \ar@{-}[r] & & \cdots &
      \ar@{-}[r] & 2^n \ar@{-}[r] & & \cdots }\] possesses a unique
    non-trivial proper closed ideal determined by the unique directed
    hereditary subgraph (the base line graph)
   \[ \hspace*{10mm}\xymatrix{  1 \ar@{-}[r]& 2\ar@{-}[r] &
      4 \ar@{-}[r] & & \cdots & \ar@{-}[r] & 2^n \ar@{-}[r] & &
      \cdots  .}\]
  \end{itemize}
The first named author would like to thank Caleb Eckhardt for pointing out
the AF algebra  example in the above list.

In \cite[Theorem 4.16]{GJ1}, it was shown that if $A$ and $B$ are
simple, unital $C^*$-algebras with one of them admitting no tracial
functionals, then the only proper non-trivial closed Lie ideal of $A
\oh B$ or $A \oop B$ is $\C (1 \ot 1)$. Using above form of Lie
normalizers of ideals, we can now generalize this result to
the following:\\

\begin{theorem}\label{thm1}
  Let $A$ be a simple $C^*$-algebra and $B$ be any $C^*$-algebra with
  one of them admitting no tracial states. If $\ota =
  \oh, \obp$ or $\oop$, then a subspace $L$ of $A \ota B$ is a closed Lie
  ideal if and only if it is of the form $$L = \left\{
  \begin{array}{ll}
    \overline{A \ot J + \C 1 \ot S}, & \ \text{if}\ A \ \text{is
      unital, and} \\ A \ota J, &
    \ \text{if}\ A \ \text{is non-unital},
  \end{array} \right. $$
  for some closed ideal $J$ in $B$ and some  subspace $S$ of  $N(J)$.
\end{theorem}

\begin{proof}
Suppose $L$ is a closed Lie ideal in $A \ota B$. By \cite[Proposition
  5.2]{sinc}, \cite[Theorem 3.13]{GJ2} and \cite[Theorem 3.8]{jk11},
every closed  ideal in $A \obp B$ is a product ideal of the form $ A \ota J$ for some
closed ideal $J$ in $B$, which admits a quasi-central approximate
identity (by \cite[Corollary 3.4]{sinc} and \cite[Proposition
  4.11]{GJ1}) and, by \Cref{tf-ts} and \Cref{commutators-c-star}, $A
\ota B$ has no tracial functionals. Thus, by \cite[Corollary 4.7 $\&$
  Lemma 4.2]{GJ1}, there exists a closed ideal $K$ in $A \ota B$ such
that $K \subseteq L \subseteq N(K)$. As noted above, $K$ is of the
form $K = A \ota J$ for some closed ideal $J$ in $B$.  And, from
\Cref{NK} \and \Cref{NK-oop}, we
have \begin{equation}\label{N-A-oop-J} N(A \ota J) = \ol{A \ot J +
    \mcal{Z}(A) \ot N(J)}.\end{equation} Since $A$ is simple,
$\mcal{Z}(A)$ is either $\C 1$, if $A$ is unital, or $\{0\}$, if $A$
is non-unital - see \cite[Lemma 2.1]{ART}. Hence, $L$ must be of the
form as in the statement.
 
 Converse is trivial. \end{proof}

\begin{corollary}\label{A-oh-BH}
  Let $A$ be a simple $C^*$-algebra,  $H$ be an infinite
  dimensional separable Hilbert space and $\ota = \oh, \obp$ or $\oop$. If $A$ is unital, then
  $A \ota B(H)$ has only three non-trivial proper closed Lie ideals,
  namely, $\C (1 \ot 1)$, the closed ideal $A \ota K(H)$ and the
  closed subspace $ A \ota K(H) + \C(1 \ot 1)$. And, if $A$ is non-unital, then
  $A \ota B(H)$ has only one non-trivial proper closed Lie ideal,
  namely,  the closed ideal $A \ota K(H)$.
\end{corollary}

\begin{proof}
  Recall that $B(H)$ admits no tracial functional. Suppose $A$ is
  unital. By \Cref{thm1}, a closed Lie ideal of $A \ota B(H)$ is of
  the form $\ol{A \ot J + \C1 \ot S}$ where $J$ is either $(0), K(H)$
  or $B(H)$ and $S$ is a subspace of $N(J)$.  We have $N((0))
  =\mcal{Z}(B(H)) = \C 1$ and, by \Cref{4A}, $N(K(H)) = \C 1 +
  K(H)$. So, for a subspace $S \subseteq \C1 +K(H)$, we have  \[
  \ol{A \ot K(H) + \C1
    \ot S} \subseteq \ol{A \ot K(H) + \C(1 \ot 1)} = A \ota K(H) + \C(1
  \ot 1),
  \]
  where we could remove the closure in last equality, for
  instance, by \Cref{NI-banach}.  So, the only non-trivial proper closed Lie
  ideals of $A \ota B(H)$ are $\C (1 \ot 1)$, $A \ota K(H)$ and $A \ota
  K(H) + \C(1 \ot 1)$.

And, when $A$ is non-unital, $A \obp K(H)$ is the only non-trivial
proper closed Lie ideal in $A \obp B(H)$, by \Cref{thm1}.  \end{proof}

If $H$ is finite dimensional, then the Lie ideals of the algebraic
tensor product $A \ot \M_n$ were identified in \cite[Corollary
  4.18]{bks}.  Using this identification, we deduce the following:\\

 \begin{theorem}
Let $A$ be a unital $C^*$-algebra and $\|\cdot\|_\alpha $ be any
algebra cross norm (for instance, $ \oop, \omin, \obp$ or $
\oh$). Then, a subspace $L$ of $A \ota \M_n$ is a closed Lie ideal if
and only if it is of the form $L = I \ot sl_n + M \otimes \C 1$ for
some closed ideal $I$ and a closed subspace $M$ of $A$ satisfying
$\overline{[I,A]} \subseteq M \subseteq N(I)$, where $sl_n :=
\big\{[t_{ij}] \in \M_n : \sum_{i} t_{ii} = 0 \big\}$.
\end{theorem}

\begin{proof}
We first show that any such $L$ is a closed Lie ideal in $A \ota \M_n$.
Since every closed ideal in a $C^*$-algebra possesses a quasi-central
approximate identity, we have $N(\ol{[I, A]}) = N(I)$ - see, for
instance, \cite[Lemma 4.2]{GJ1}. So, $\overline{[I,A]} \subseteq M
\subseteq N(I) = N(\ol{[I, A]})$ which implies that $M$ is a Lie ideal
in $A$.  Now, let $x \in I, m \in M, c \in sl_n, a \in A$ and $b
\in \M_n$. Then, \begin{eqnarray*} [x \ot c + m \ot 1 , a \ot b] & = &
  x a \ot cb - ax \ot bc + ax \ot cb - ax \ot cb + [m,a] \ot b \\ & =
  & [x, a] \ot cb + a x \ot [c, b] + [m,a] \ot b \\ & \in & [I, A] \ot
  \M_n + I \ot sl_n + (M \cap I) \ot \M_n \\ & \subseteq & [I, A]
  \ot sl_n + [I, A] \ot \C1 + I \ot sl_n + M \ot \C 1 \\ & \subseteq & I \ot
  sl_n + M \ot \C 1,
\end{eqnarray*} 
where in the third last step, we have used the fact that $M \subseteq
N(I)$ and that $M$ is a Lie ideal, and the second last step follows because
$\M_n = sl_n \oplus \C 1 $.

It now remains to show that $L$ is closed as well. By
\Cref{A-ot-Mn-closed}, $A \ot \M_n$ is complete with respect to
$\|\cdot\|_{\alpha}$.  Let $\{a_k + b_k\ot 1\} $ be a Cauchy sequence
in $I \ot sl_n + M \otimes \C 1$, where $a_k \in I \ot sl_n $ and $b_k
\in M $. Let $\tau: A \ot \M_n \ra A$ be the $A$-valued trace map given
by $\tau(\sum_{ij} x_{ij} \ot E_{ij}) = \sum_i x_{ii}$. As seen in the
proof of \Cref{A-ot-Mn-closed}, $\tau$ is continuous; so, $ I \ot
sl_n = \ker(\tau) \cap (I \ot \M_n)$ is closed. In particular, since
$\{a_k + b_k\ot 1\}$ converges in $A \ot \M_n$ and $\tau (a_k + b_k\ot
1)= nb_k$, the sequence $\{b_k\}$ converges in $M$.  Thus, the sequence
$\{a_k\}$ must converge in $I \ot sl_n$, and $L$ is closed.

Conversely, suppose $L$ is a closed Lie ideal in $A \ota \M_n$. By
\cite[Corollary 4.18]{bks}, there exists an ideal $I$ and a subspace
$M$ of $A$ such that ${[I,A]} \subseteq M \subseteq N(I)$ and $L =
     {I \ot sl_n + M \otimes \C 1}$.  Now, we
     have $$ \overline{[\bar{I}, A]} = \overline{[I,A]} \subseteq
     \bar{M} \subseteq N(\bar{I})$$ and,  $L$ being closed,
$$  L = \ol{L} =  \overline {{I} \ot sl_n + M \otimes
  \C 1} \\
= \overline {\bar{I} \ot sl_n + \bar{M} \otimes
  \C 1} \\
= {\bar{I} \ot sl_n + \bar{M} \otimes
  \C 1},$$
 where the last equality holds because ${\bar{I} \ot sl_n + \bar{M} \otimes
  \C 1}$ is closed, as seen above. \end{proof}

\section{Closed Lie ideals of $A \omin C_0(X)$}
Based on the techniques of \cite{mar}, in \cite{GJ1}, all closed Lie
ideals of $A \omin B$ were identified for any unital simple
$C^*$-algebra $A$ with at most one tracial state and any unital
commutative $C^*$-algebra $B$.  Improvising the same techniques
appropriately again, in this section, we identify all the closed Lie
ideals of $A\omin B$ for simple $C^*$-algebra $A$ with at most one
tracial state and any commutative $C^*$-algebra $B$.

Throughout this section, $X$ will denote a locally compact Hausdorff
space. For every closed subset $F$ of $X$, the symbol $J(F)$ will
denote the closed ideal in $C_0(X)$ given by $J(F) = \{ f \in C_0(X):
f(F) = (0)\}$. Recall that this gives a bijective correspondence
between all the closed subsets of $X$ and all the closed ideals of
$C_0(X)$ (see \cite[Theorem 1.4.6]{kani}).  We will depend heavily on
the fact that for any $C^*$-algebra $A$ the canonical map from
$A\otimes C_{0}(X)$ into $C_0(X,A)$ extends to an isometric
$*$-isomorphism from $A\omin C_{0}(X)$ onto $C_0(X,A)$ (see
\cite[Proposition 1.5.6]{kani} and \cite[Theorem 4.14]{tak}). Under
this identification, we observe the following:

\begin{proposition}\label{J-tilde-F}
Let $A$ be a simple $C^*$-algebra. Every closed ideal of $C_0(X, A)$
is of the form $$\tilde{J}(F):= \{ f \in C_{0}(X, A) : f(x)=0
\ \text{for\ all} \ x \in F \} $$ for some closed subset $F$ of $X$.
\end{proposition}

\begin{proof}
Since a commutative $C^*$-algebra is nuclear, by
\cite[Theorem 3.1]{GJ1}, every closed ideal in $A\omin C_{0}(X)$ is a
product ideal of the form $J=A \omin J(F)$ for some closed subset $F$
of $X$. We show that $J$ corresponds to $\tilde{J}(F)$.

Clearly, $J\subseteq \tilde{J}(F)$, and for the equality it suffices
to show that $J$ is dense in $\tilde{J}(F)$. Let $f \in
\tilde{J}(F)$ and $\epsilon > 0$. Since $C_c(X,A)$ is dense in
$C_0(X,A)$, there exists a $g \in C_c(X,A)$ such that $\norm{f -
  g} < \epsilon $.

Let $K:= \text{supp}(g)$ and, for each $a \in A$ and $r>0$, let
$B_r(a):=\{ b \in A: \| b - a\| < r \}$ and $B_r^{\times}(a):=
B_r(a)\setminus\{0\}$. Since $\|g(x) \| = \|f(x) -g(x)\| < \epsilon $
for every $x \in F$, the collection $\lbrace g^{-1}\left(B_{\epsilon
}^{\times}(g(x)) \setminus \overline{g(F)} \right): x \in K\setminus F
\rbrace$ $\cup \lbrace  g^{-1}({ B_{\frac{3\epsilon}{2}}(0))} \rbrace$
forms an open cover of the compact set $K$.  Fix a finite subcover,
say, $ \lbrace g^{-1}({ B_{\frac{3\epsilon}{2}}(0))} \rbrace$ $ \cup
\lbrace g^{-1}\big( B_{\epsilon}^{\times}(g(x_i)) \setminus
\overline{g(F)}\big)$ $: 1 \leq i \leq n \rbrace$. Since $X$ is
locally compact and Hausdorff, there exists a partition of unity on
$K$ sub-ordinate to this finite subcover, i.e., there exists a family
$\{ h_i: 0 \leq i \leq n\} \subset C_c(X)$ such that $0 \leq h_i \leq
1$ for all $ 0 \leq i \leq n$, $\text{supp}(h_0) \subseteq U_0 :=
g^{-1}( B_{\frac{3\epsilon}{2}}(0))$, $\text{supp}(h_i) \subseteq U_i
:= g^{-1}\big({ B_{\epsilon}^{\times}(g(x_i)) \setminus
  \overline{g(F)} \big)}$ for all $1 \leq i \leq n$ and $\sum_{i =
  0}^n h_i = 1$ on $K$ (see \cite[Theorem 2.13]{rud}). Clearly, $U_i \cap F = \emptyset$ for all $ 1 \leq i \leq n$.

Let $ V = \{ x \in X: \sum_{i = 0}^n h_i(x) < 3/2\} $.  Then $V \cap
(\cup_{i=0}^{n} U_i)$ is an open set containing $K$.  Using Urysohn's
Lemma, pick an $ h' \in C_c(X)$ such that $0 \leq h' \leq 1 $, $h'(K)
=\{1\}$ and $\text{supp}(h') \subset V \cap (\cup_{i=0}^{n} U_i)$.
Then, for $ h_i' := h' h_i \in C_c(X)$, $\text{supp}(h_i') \subset V \cap U_i $
because $\text{supp}(h_i) \subset U_i $ and $\text{supp}(h') \subset V
$.  Also, for $x \in K$ we have
\[
\sum_{i = 0}^n h_i'(x) = \sum_{i =
  0}^n h'(x) h_i(x) = \sum_{i = 0}^n h_i(x) = 1.
\]
Moreover,  $0
\leq \sum_{i = 0}^n h_i' \leq 3/2 $ because $\sum_{i = 0}^n h_i'(x) =
\sum_{i = 0}^n h'(x) h_i(x) \leq \sum_{i = 0}^n h_i(x) = 3/2$ for $x
\in V \cap (\cup_{i=0}^{n} U_i)$ and $\text{supp}(\sum_{i = 0}^n h_i')
\subset V \cap (\cup_{i=0}^{n} U_i)$.

Note that for $1 \leq i \leq n$, $U_i$, and hence $V \cap U_i$ is
disjoint from $F$, so $h_i'\in J(F)$ and $\sum_{i=1}^{n} g(x_i)
\otimes h_i' \in A \otimes J(F)$.  Also, for  each $x \in
X$, we have $g(x) = g(x) \sum_{i=0}^nh_i'(x)$ because when $x \in
K$, then  $\sum_{i = 0}^n h_i' (x) = 1$ and when $x \in K^c$,
then $g(x) = 0$. Fix an $x_0 \in K^c$. Then, we observe that 
\vspace*{-2mm}
\begin{eqnarray*}
  \| g(x) - \sum_{i = 1}^n h_i'(x)  g(x_i)\| & =& \| g(x) \sum_{i = 0}^n
  h_i' (x) - \sum_{i = 0}^n h_i' (x) g(x_i)\| \qquad (\text{since}
  \ g(x_0) = 0)\\ & \leq & \sum_{i = 0}^n  \|g(x)  - g(x_i)\| h_i' (x) \\
 & = & \sum_{i: x \in U_i \cap V} \| g(x)  - g(x_i)\| h_i' (x) \\
& <& \frac{3 \epsilon}{2}
\end{eqnarray*}
for all $x \in X$. In
particular, $\| f - \sum_{i=1}^{n} g(x_i) \otimes h_i'\| < \frac{5
  \epsilon}{2}$, implying that $J$ is dense in $\tilde{J}(F)$.
\end{proof}

Recall that for a $C^*$-algebra $A$ with $\mathcal{T}(A) \neq
\emptyset$, $sl(A):= \cap\{\ker (\varphi) : \varphi \in \mcal{T}(A)\}$ is
a closed Lie ideal in $A$ and so is $\overline{sl(A) \ot J(F)}$ in $A
\omin C_0(X)$. We show what it corresponds to in $C_0(X, A)$.

\begin{lemma}\label{L-tilde-F}
For a $C^*$-algebra $A$ with $\mathcal{T(A)} \neq \emptyset$ and a closed subset
$F$ of $X$, the closed Lie ideal $L(F) := \overline{sl(A) \ot J(F)}$ in
$A \omin C_0(X)$ corresponds to the closed Lie ideal
\[
\tilde{L}(F): = \lbrace f \in C_0(X,A) :\, f(x)=0 \ \text{for\ all}\ x \in F\ \text{and}\
\varphi\circ  f= 0\ \text{for\ all}\ \varphi \in \mathcal{T}(A) \rbrace
\]
in $C_0(X,A)$.
\end{lemma}

\begin{proof}
Under the canonical $*$-isomorphism between $ A \omin C_0(X)$ and $C_0(X,
A)$, the closed Lie ideal ${L}(F)$ is mapped onto a closed Lie ideal
in $\widetilde{L}(F)$. It just remains to show that the image is dense
in $\widetilde{L}(F)$.

Let $f \in \widetilde{L}(F)$ and $\epsilon > 0$.  Then, $f \in
\widetilde{J}(F)$, and as in the proof of \Cref{J-tilde-F}, there
exist a $g \in C_c(X, A)$ with $\| f - g\| < \epsilon $, and finite
sets $\{x_1, \ldots, x_n\} \subseteq K \setminus F$ (where $K:=
\text{supp}(g)$) and $\{h_1', \ldots, h_n'\} \subseteq J(F)$ such that
$0 \leq \sum_{i=1}^{n} h_i' \leq 3/2$ (because $0 \leq \sum_{i=0}^{n}
h_i' \leq 3/2$ and $0 \leq h_0' \leq 1$) and $\| f - \sum_{i = 1}^n
g(x_i) \ot h_i'\| < \frac{5 \epsilon}{2}$. Since $\|f(x_i) - g(x_i)\|
< { \epsilon}$ for all $1 \leq i \leq n$, we see that
\[
\| f - \sum_{i
  = 1}^n f(x_i) \ot h_i'\| \leq \| f - \sum_{i = 1}^n g(x_i) \ot
h_i'\| + \| \sum_{i = 1}^n g(x_i)\ot h_i' - \sum_{i = 1}^n f(x_i) \ot
h_i'\| < \frac{8\epsilon}{2}.
\]
Further, since $sl(A) =
\cap_{\varphi \in \mcal{T}(A)} \ker (\varphi)$ and $f \in
\widetilde{L}(F)$, it readily follows that $\sum_{i = 1}^n f(x_i) \ot
h_i' \in sl(A) \ot J(F)$ and we are done.  \end{proof}

Applying \Cref{L-tilde-F}, the proof of \cite[Proposion 5.3]{GJ1}
works verbatim to yield  the following:

\begin{proposition}\label{L-closed}
Let $A$ be a unital C*-algebra with $\mcal{T}(A) \neq \emptyset$. Then,
a subspace of the form $L = \overline{sl(A) \ot J} + \mathbb{C}1
\ot S$, where $S$ is a closed subspace  and $J$
is a closed ideal in $C_0(X)$, is a closed Lie ideal in $A \omin C_0(X)$
\end{proposition}

The next step is to identify the Lie normalizer of any closed ideal of
$A \omin C_0(X)$, which required an appeal to the Tietze Extension
Theorem in \cite{mar} and \cite{GJ1}. Since we are in a locally
compact setting, we will need a slightly different version of Tietze
Extension Theorem.  Before providing the precise statement, we first
recall a version of Urysohn's Lemma, that suits us, a proof of which
can be deduced easily from \cite[Theorem 2.12]{rud}.

\begin{theorem}[Urysohn's Lemma]\label{urysohn}
Let $X$ be a locally compact Hausdorff space, $V$ be an open set in
$X$ and $K$ be a compact subset of $X$ such that $K \subset U$. Then
there exists an $f \in C_c\left(X, [0,1]\right)$ such that $f(K) =\{ 1\}$ and
$supp(f) \subseteq V$.

In particular, if $K_1$ and $K_2$ are two disjoint compact subsets of
$X$, then for each $r > 0$, there exists an $f \in C_c(X, [-r, r])$
such that $f(K_1) = \{-r\}$ and $f(K_2) = \{r\}$.
\end{theorem}

Since we are unaware of any standard reference, adapting the proof
of \cite[Theorem 20.4]{rud}, we  obtain the following version of
Tietze Extension Theorem.
\begin{theorem}[Tietze Extension Theorem]\label{tietze}
Let $X$ be a locally compact Hausdorff space and $F$ be a closed
subspace of $X$. Then, for each $f \in C_0(F)$ there exists an
$\tilde{f} \in C_0(X)$ such that $\tilde{f}_{|_{F}} = f$.
  \end{theorem}
\begin{proof}
  Without loss of generality, assume that $f$ is real valued and that
  $-1 \leq f \leq 1$.  Note that $F_0:=\{x \in F: f(x) = 0\}$ is
  closed in $F$ and hence in $X$. So $X \setminus F_0$ is open in
  $X$. Now, since $f \in C_0(F)$, the sets $F^+:=\{ x \in F: f(x) \geq
  1/3\}$ and $F^-:=\{ x \in F: f(x) \leq -1/3\}$ are compact in $F$
  and hence in $X$ as well. Also, $F^+$ and $ F^-$ are disjoint (and
  contained in $W$). So, by above version of Urysohn's Lemma, there
  exists an $f_1 \in C_c(X, [-1/3, 1/3])$ such that $f(F^-) = \{
  -1/3\}$ and $f(F^+) = \{ 1/3\}$. Thus, $$f - (f_1)_{|{F}} \in
  C_0(F),\, |f - f_1| \leq \frac{2}{3}\ \text{on}\ F\ \text{ and}\ |f_1| \leq
  \frac{1}{3}.$$ Repeating the argument for $f-f_1$, there exists an $f_2 \in
  C_c(X, [\frac{-1}{3} \cdot \frac{2}{3}, \frac{1}{3} \cdot
    \frac{2}{3}])$ such that $$f - (f_1 - f_2)_{|_{F}} \in C_0(F),\, |f
  - f_1-f_2| \leq \left(\frac{2}{3}\right)^2\ \text{on}\ F\ \text{
    and}\ |f_2| \leq \frac{1}{3} \cdot \frac{2}{3}.$$ Continuing the
  process, we obtain a sequence $\{f_n\} \subset C_c(X)$ such that
$$|f - f_1 - f_2 - \cdots - f_n| \leq
  \left(\frac{2}{3}\right)^n\ \text{ on}\ F\ \text{ and}\ |f_n| \leq
  \frac{1}{3} \cdot \left(\frac{2}{3}\right)^{n-1}\ \text{for all}
  \ n.$$ Thus, $\sum_n f_n$ is a convergent series in $C_0(X)$ and if
  $\tilde{f}$ denotes its sum, then we have $f  =\tilde{f}$ on $F$,
  as was required. \end{proof}

We can now identify the Lie normalizer of any closed ideal of $A \omin
C_0(X)$, which will be crucial in the identification of closed Lie
ideals later. We adapt the technique used in \cite{GJ1} and \cite{mar}
to fit our requirements.

\begin{lemma}\label{lie-normalizer}\label{6A}
Let $A$ be a simple  $C^*$-algebra and $I$ be any closed ideal
in $A \omin C_0(X)$. Then, 
\[
N(I)= \left\{
\begin{array}{ll}
  I + \mathbb{C}1 \otimes C_0(X), & \text{if}\ $A$ \ \text{is unital, and}\\
  I, &  \text{if}\ $A$ \ \text{is non-unital}.
\end{array}
\right.
\]
\end{lemma}
\begin{proof}
Since $C_0(X)$ is nuclear and $A$ is simple, by \cite[Theorem
  3.1]{GJ1}, $I$ is of the form $A \omin J$ for some closed ideal $J$
in $C_0(X)$. Clearly, $A \omin J+ \C 1 \otimes C_0(X) \seq N(A \omin
J)$ if $A$ is unital and $A \omin J \seq N(A \omin J)$ when $A$ is non-unital.

Now, let  $ f \in N(A \omin J) \seq C_0(X, A)$. Then,  $[f, g] \in A
\omin J$ for all $g \in C_0(X, A)$.  Since each singleton $\{x\}$ is
compact and $X$ is locally compact and Hausdorff, by Uryshon's Lemma,
there exists an $f_x \in C_c(X)$ such that $f_x(x) = 1 $ (see
\Cref{urysohn}). Now, for each $a \in A$, Define $\tilde{f}_{x,a}: X
\ra A$ by $ \tilde{f}_{x,a} (y) = f_x(y)a$ for $y \in X$. Let $F$ be
the closed set in $X$ that determines the closed ideal $J$, i.e.,
$J(F) = J$. Then, $\tilde{f}_{x,a} \in C_c(X, A) \subseteq C_0(X, A)$
and we have $ 0 = [f, \tilde{f}_{x,a}](x) = f(x) a - a f(x)$ for all
$a \in A$ and $x \in F$.  Thus, $f(x) \in \mcal{Z}(A)$ for all
$x \in F$.

If $A$ is unital, then $\mcal{Z}(A) = \C 1$. In particular,
$g:=f_{|_F}$ can be thought of as a scalar valued function on
$F$. Since $f \in C_0(X, A)$, we have $g \in C_0(F)$. So, by
\Cref{tietze}, there exists a $\tilde{g}\in C_0(X)$ such that
$\tilde{g}_{|_{F}} = g$.  Under the identification $C_0(X) = \C1
\otimes C_0(X)$, we have $\tilde{g} \in \C 1 \ot C_0(X)$. And, since
$f (x) - \tilde{g} (x) = 0$ for all $ x \in F$, by \Cref{J-tilde-F},
we have $f - \tilde{g} \in A \omin J(F)$. Hence, $f \in A \omin J +
\mathbb{C}1 \otimes C_0(X)$ so that $N(A \omin J) =  A \omin J +
\mathbb{C}1 \otimes C_0(X)$.

And, if $A$ is non-unital, then $\mcal{Z}(A) = (0)$ - see \cite[Lemma
  2.1]{ART}. So, $ f \in A\omin J$ and hence  $N(A \omin J) = A
\omin J$.
\end{proof}

We conclude this article with  the promised characterization of closed Lie ideals
of $A \omin C_0(X)$, which is also a generalization of \cite[Theorem 3.1]{mar}
and \cite[Theorem 5.6]{GJ1}.

\begin{theorem}\label{l-ideal}
  Let $A$ be a simple $C^*$-algebra and $B$ be a commutative
  $C^*$-algebra. Then a subspace $L$ of $A \omin B$ is a closed Lie
  ideal if and only if
  \begin{eqnarray}\label{l-ideal-eqn-unital}
 &&   \ol{sl(A) \ot J} \subseteq L \subseteq A \omin J + \C 1  \ot B\ \text{if A is unital, and} \\
 && \ol{sl(A) \ot J} \subseteq L \subseteq A \omin J  \ \text{if A is non-unital} \label{l-ideal-eqn-non-unital}
  \end{eqnarray}
  for some closed ideal $J$ in $B$. Further, if $A$ is unital and admits a
      unique tracial state, and $L$ is a closed subspace of $A \omin
      B$ satisfying (\ref{l-ideal-eqn-unital}), then $L$ is of the form $L
      = \ol{sl( A) \ot J} + \C 1 \otimes S$ for some closed subspace
      $S$ in $B$.
\end{theorem}
\begin{proof}
  Let $B = C_0(X)$ for some locally compact Hausdorff space $X$.  Let
  $L$ be a subspace of $A \omin C_0(X)$.  Since $A$ is simple, by
  \cite[corollary 4.7]{GJ1} and \cite[Theorem 3.1]{GJ1}, $L$ is a
  closed Lie ideal if and only if there exists a closed ideal $J$ in
  $C_0(X)$ such that $$\overline{[A \omin C_0(X), A \omin J]}
  \subseteq L \subseteq N(A \omin J).$$ We first show that $ \ol{[A
      \omin C_0(X), A \omin J]} = \overline{sl(A) \ot J}.$ Indeed, by
  \cite[Lemma 4.2]{GJ1}, we have
  $$\overline{[A \omin C_0(X),A \omin J]} = \overline{[A \omin J, A
      \omin J]},$$ then, by commutativity of $C_0(X)$, it can be easily
  deduced that $\overline{[A \omin J ,A \omin J]}=
  \overline{\overline{[A ,A]} \ot J}$, and,   by \cite[Theorem 2.9]{cnp},
  we have $sl(A)=\overline{[A,A]}$.

  On the other hand, by \Cref{lie-normalizer}, we have $N( A \omin J)
  = A \omin J + \C 1 \ot C_0(X)$ (resp., $A \omin J$) if $A$ is unital (resp., non-unital), thereby establishing
  (\ref{l-ideal-eqn-unital}) and (\ref{l-ideal-eqn-non-unital}).

  Finally, suppose that $A$ admits a unique tracial state and that $L$
satisfies (\ref{l-ideal-eqn-unital}).  Thanks to  \Cref{L-closed}, it just remains to show that $A \omin J +
\mathbb{C}1 \otimes C_0(X)= \overline{sl(A) \ot J} + \mathbb{C}1
\otimes C_0(X)$. Indeed, since $\mathcal{T}(A)$ is singleton, we have $A= \mathbb{C}1
\oplus sl(A)$;  so
\begin{eqnarray*}
  A \omin J + \mathbb{C}1 \otimes C_0(X) & = &
(\mathbb{C}1 \oplus sl(A)) \omin J + \mathbb{C}1 \otimes
C_0(X) \\
&  = &  \overline{sl(A) \ot J} +\mathbb{C}1
\otimes C_0(X) .
\end{eqnarray*}  
\end{proof}
In particular, this yields the following structure of closed Lie ideals.
\begin{corollary}\label{A-omin-B}
 Let $A$ be a simple  $C^*$-algebra  and $B$ be a commutative $C^*$-algebra.
 \begin{enumerate}
   \item If $A$ is unital and admits at most one tracial state, then a  subspace $L$ of $A \omin B$ is a
closed Lie ideal if and only if
$$  L = \left\{
 \begin{array}{ll}
      \overline{sl(A) \ot J} + \C 1 \ot S, & \mathit{if}\ \mcal{T}(A)
     \neq \emptyset,\ \mathit{and} \\ \overline{A \otimes J + \C 1 \ot
       S}, & \mathit{if}\ \mcal{T}(A) = \emptyset,
 \end{array}
 \right.    $$
 for some closed ideal  $J$ and closed subspace $S$ in $B$.
\item If $A$ is non-unital with 
  $\mcal{T}(A) = \emptyset$ , then a subspace $L$ of $A \omin B$ is a
  closed Lie ideal if and only  $L$ is a closed ideal.
 \end{enumerate}
 \end{corollary}

\begin{proof}
 (1) We only need to prove the necessity part. Suppose $\mcal{T}(A)$ is a
  singleton. Then, by \Cref{l-ideal}, we obtain the desired form for
  $L$.
  \vspace*{1mm}
  
 And, if $\mcal{T}(A) = \emptyset$, then $sl(A) = \ol{[A, A]} = A$ -
 see \Cref{tf-ts}. In particular, by \Cref{l-ideal}, $L$ must be of
 the form $L = \overline{A \ot J + \C 1 \ot S}$ for some closed
 subspace $S$ of $B$.
 \vspace*{1mm}
 
 (2) follows immediately from \Cref{l-ideal}. 
\end{proof}

\end{document}